\documentclass[12pt]{amsart}

\usepackage[utf8]{inputenc}
\usepackage[T1,T2A]{fontenc}
\usepackage[english]{babel}
\usepackage{amsfonts}
\usepackage{amsbsy}
\usepackage{amssymb}
\usepackage{fixltx2e,graphicx}
\usepackage{mathrsfs}
\usepackage{hyperref}
\usepackage{longtable}
\usepackage{enumitem}
\usepackage{comment}

\newcommand{\Hom}{\operatorname{Hom}}

\newcommand{\SL}{\operatorname{SL}}

\newcommand{\Aut}{\operatorname{Aut}}
\newcommand{\ord}{\operatorname{ord}}

\newcommand{\ZZ}{\mathbb Z}
\newcommand{\QQ}{\mathbb Q}

\newcommand{\GG}{\mathbb G}
\newcommand{\KK}{\mathbb K}

\newtheorem{theorem}{Theorem}

\newtheorem{proposition}{Proposition}

\newtheorem{corollary}{Corollary}

\newtheorem*{question*}{Question}

\theoremstyle{definition}

\newtheorem{definition}{Definition}

\theoremstyle{remark}

\newtheorem{remark}{Remark}

\makeatletter
\newcommand{\newsectionstyle}{%
  \renewcommand{\@secnumfont}{\bfseries}
  \renewcommand\section{\@startsection{section}{2}%
    \z@{.5\linespacing\@plus.7\linespacing}{-.5em}%
    {\normalfont\bfseries}}%
}
\let\oldsection\section
\let\old@secnumfont\@secnumfont
\newcommand{\originalsectionstyle}{%
  \let\@secnumfont\old@secnumfont
  \let\section\oldsection
}
\makeatother

\begin{document}

\renewcommand{\proofname}{Proof}
\renewcommand{\abstractname}{Abstract}
\renewcommand{\refname}{References}
\renewcommand{\figurename}{Figure}
\renewcommand{\tablename}{Table}

\title[On the existence of $B$-root subgroups]
{On the existence of $B$-root subgroups on affine spherical varieties}

\author{Roman Avdeev and Vladimir Zhgoon}


\address{%
{\bf Roman Avdeev} \newline\indent HSE University, Moscow, Russia}

\email{suselr@yandex.ru}

\address{%
{\bf Vladimir Zhgoon} \newline\indent Scientific Research Institute for System Analysis of the Russian Academy of Sciences, Moscow, Russia \newline\indent HSE University, Moscow, Russia}

\email{zhgoon@mail.ru}




\begin{abstract}
Let $X$ be an irreducible affine algebraic variety that is spherical with respect to an action of a connected reductive group~$G$.
In this paper we provide sufficient conditions, formulated in terms of weight combinatorics, for the existence of one-parameter additive actions on~$X$ normalized by a Borel subgroup $B \subset G$.
As an application, we prove that every $G$-stable prime divisor in~$X$ can be connected with the open $G$-orbit by means of a suitable $B$-normalized one-parameter additive action.
\end{abstract}

\maketitle

\newsectionstyle

\section{}

Let $X$ be an irreducible algebraic variety over an algebraically closed field $\KK$ of characteristic zero equipped with an action of a connected reductive algebraic group~$G$.
Every nontrivial regular action on~$X$ of the additive group $\GG_a = (\KK,+)$ induces an algebraic subgroup in the automorphism group $\Aut(X)$, called a \textit{$\GG_a$-subgroup}.
For an arbitrary $\GG_a$-subgroup $H$ on~$X$, every nonzero element of the one-dimensional Lie algebra $\mathrm{Lie}(H)$ naturally defines a locally nilpotent derivation (LND for short) $\partial$ on the algebra of regular functions~$\KK[X]$, and in the case of quasi-affine~$X$ the subgroup $H$ can be recovered by taking the exponent of~$\partial$.

In the present paper we are interested in $\GG_a$-subroups on~$X$ normalized by the action of a Borel subgroup $B \subset G$.
Following~\cite{AA}, we call such $\GG_a$-subgroups \textit{$B$-root subgroups} on~$X$.
For every $B$-root subgroup $H$ on~$X$, the adjoint action of $B$ on $\mathrm{Lie}(H)$ reduces to multiplying by a character of~$B$, which we denote by~$\chi_H$ and call the \textit{weight} of~$H$.
If~$\partial$ is an LND on~$\KK[X]$ corresponding to~$H$ then $\partial$ is also normalized by~$B$ with the same weight~$\chi_H$.

\section{} \label{S_types_of_colors}

From now on we assume that $X$ is a \textit{spherical} $G$-variety, that is, $X$ is normal and possesses an open $B$-orbit.
Let $\mathcal D^B$ (resp.~$\mathcal D^G$) denote the finite set of all $B$-stable (resp. $G$-stable) prime divisors in~$X$.
Elements of the set $\mathcal D = \mathcal D^B \setminus \mathcal D^G$ are traditionally called \textit{colors} of~$X$.

We now recall the well-known division of colors into three types (see~\cite{BORB}, \cite{BBORB}).
Fix an arbitrary color $D \in \mathcal D$.
Then one can always choose a minimal parabolic subgroup $Q \supset B$ in~$G$ such that $QD \ne D$.
For every subgroup $F \subset Q$, let $\overline F$ be its image in $Q/Q_r$, where $Q_r$ is the solvable radical of~$Q$.
Choose an arbitrary point $z$ in the open $B$-orbit in~$X$ and let $Q_z$ be the stabilizer of~$z$ in~$Q$.
Observe that $Qz \cap D \ne \varnothing$.
Since $Q_r \subset B$, the natural morphism
\begin{equation} \label{mor_Q_r}
Qz \simeq Q / Q_z \to Q / (Q_zQ_r) \simeq \overline{Q} / \overline{Q}_z
\end{equation}
of factorizing by~$Q_r$ induces a codimension-preserving bijection between $B$-orbits in $Qz$ and $\overline B$-orbits in $\overline Q/ \overline Q_z$.
In particular, $\overline Q/ \overline Q_z$ contains an open $\overline B$-orbit.
Since $\overline Q$ is isomorphic to either $\SL_2$ or $\mathrm{PSL}_2$, there are the following three possibilities for~$Q_z$.

Type~$(U)$: $\overline Q_z$ contains a maximal unipotent subgroup in~$\overline Q$.
In this case, $Qz \setminus Bz$ is a single $B$-orbit of codimension~$1$, which coincides with $Qz \cap D$.

Type~$(T)$: $\overline Q_z$ is a maximal torus in~$\overline Q$.
In this case, $Qz \setminus Bz$ contains two $B$-orbits of codimension~$1$ and one of them coincides with $Qz \cap D$.

Type~$(N)$: $\overline Q_z$ is the normalizer of a maximal torus in~$\overline Q$.
In this case, $Qz \setminus Bz$ is a single $B$-orbit of codimension~$1$, which coincides with $Qz \cap D$.

It is known that the above-defined type does not depend on the choice of the minimal parabolic subgroup $Q \supset B$ satisfying $QD \ne D$ (see~\cite[Prop.~1]{BBORB}).
This makes the type of every color in~$X$ well defined.

\begin{remark}
The above-defined types $(U)$, $(T)$, $(N)$ of colors in~$X$ coincide with the types $b$, $a$, $a'$, respectively, in the notation of Luna (see~\cite[\S\S\,2.7,\,3.4]{Lu97} or~\cite[\S\,30.10]{Tim}).
\end{remark}

\section{}

As follows from~\cite[Prop.~1.6]{AA}, for every $B$-root subgroup $H$ on~$X$ there is at most one divisor $D \in \mathcal D^B$ such that $HD \ne D$.
The following result generalizes~\cite[Cor.~4.25]{AA} where the case of affine~$X$ was considered.

\begin{proposition}\label{UN_colors}
Suppose a $B$-root subgroup $H$ on $X$ satisfies $HD \ne D$ for some $D \in \mathcal D^B$.
Then $D$ is either $G$-stable or is a color of type~$(T)$.
\end{proposition}

\begin{proof}
Assume that $D$ is a color of type $(U)$ or~$(N)$ and choose a minimal parabolic subgroup $Q \supset B$ satisfying $QD \ne D$.
Then, in the notation of~\S\,\ref{S_types_of_colors}, the orbit $Qz$ splits into two $B$-orbits $O_1 = Bz$ and $O_2 = O \cap D$.
In this case, it follows from the discussion in~\cite[\S\,1.5]{AA} that the set $Qz$ is $H$-stable and moreover each $H$-orbit in $Qz$ is isomorphic to the affine line $\mathbb A^1$ and meets $O_2$ in exactly one point.
For every subgroup $F \subset Q$, let $\widetilde F$ denote its image in $Q/Q_u$, where $Q_u$ is the unipotent radical of~$Q$.
Similarly to~(\ref{mor_Q_r}), the morphism
\begin{equation}
\varphi \colon Qz \simeq Q / Q_z \to Q / (Q_zQ_u) \simeq \widetilde{Q} / \widetilde{Q}_z
\end{equation}
of factorizing by~$Q_u$ induces a codimension-preserving bijection between $B$-orbits in $Qz$ and $\widetilde B$-orbits in~$\widetilde Q/\widetilde Q_z$.
Since the actions of $H$ and $Q_u$ on $Qz$ commute, the former descends to a nontrivial action of~$H$ on $\widetilde Q/\widetilde Q_z$ normalized by the action of~$\widetilde B$.
In particular, there are exactly two $\widetilde B$-orbits $\varphi(O_1)$, $\varphi(O_2)$ in $\widetilde Q / \widetilde Q_z$ and every $H$-orbit is isomorphic to $\mathbb A^1$ and meets $\varphi(O_2)$ in exactly one point.
Further we consider the cases of types $(U)$ and $(N)$ separately.
The condition on each of these types is reformulated in view of the fact that $\overline Q$ is the quotient of $\widetilde Q$ by its connected center.

Type~$(U)$: $\widetilde Q_z$ contains a maximal unipotent subgroup in~$\widetilde Q$.
Then $\widetilde Q/\widetilde Q_z$ has a point fixed by the unipotent radical $\widetilde B_u$ of~$\widetilde B$.
Since $\widetilde B_u$ is normalized by $\widetilde B$ and commutes with the action of~$H$, the set of $\widetilde B_u$-fixed points in $\widetilde Q/\widetilde Q_z$ is stable with respect to both groups $\widetilde B$ and~$H$ and hence coincides with the whole~$\widetilde Q/\widetilde Q_z$.
Consequently, $\widetilde B_u$ acts trivially on~$\widetilde Q/\widetilde Q_z$, a contradiction.

Type~$(N)$: $\widetilde Q_z$ contains a subgroup $\widetilde Q'_z$ of index~$2$ that is the preimage of the connected component of the identity in~$\overline Q_z$.
Then the natural morphism $\psi \colon \widetilde Q / \widetilde Q'_z \to \widetilde Q / \widetilde Q_z$ is an unramified two-fold covering.
Moreover, the set $\psi^{-1}(\varphi(O_1))$ is an open $\widetilde B$-orbit in $\widetilde Q/\widetilde Q'_z$ and the set $\psi^{-1}(\varphi(O_2))$ splits into two $\widetilde B$-orbits of codimension~$1$, which we denote by $D_1$ and~$D_2$.
Now let $y = \varphi(z)$ and $\psi^{-1}(y) = \lbrace y_1, y_2 \rbrace$.
Since $Hy \simeq \mathbb A^1$, the set $\psi^{-1}(Hy)$ is a disjoint union of two components $Y_1$ and~$Y_2$, each of which maps isomorphically onto~$Hy$.
Without loss of generality we assume that $y_i \in Y_i$ for $i=1,2$.
Let $b \in \widetilde B$ be such that $by_1 = y_2$.
Then $b \in \widetilde B_z$ and hence $by_2 = y_1$.
As $Hy$ is $\widetilde B$-stable, the action of~$b$ interchanges~$Y_1$ and~$Y_2$.
On the other hand, the set $\psi^{-1}(Hy \cap \varphi(O_2))$ consists of two points belonging to different $\widetilde B$-orbits, a contradiction.
\end{proof}

In the terminology of~\cite[\S\,4.2]{AA}, a $B$-root subgroup $H$ on~$X$ is called \textit{vertical} if it preserves the open $B$-orbit and \textit{horizontal} otherwise.
If $H$ is horizontal and $HD \ne D$ for some $D \in \mathcal D^B$ then we say that $H$ \textit{moves}~$D$.
In accordance with Proposition~\ref{UN_colors}, horizontal $B$-root subgroups can be divided into two types.

\begin{definition}
Let $H$ be a horizontal $B$-root subgroup on~$X$ and let $D \in \mathcal D^B$ be such that $HD \ne D$.
If $D \in \mathcal D^G$ then we call $H$ \textit{toroidal}.
If $D$ is a color of type~$(T)$ then we call $H$ \textit{blurring}.
\end{definition}

\section{}

For every subset $\mathcal F \subset \mathcal D$, put $D_{\mathcal F} = \bigcup_{D \in \mathcal F} D$, $X_{\mathcal F} = X \setminus D_{\mathcal F}$ and let $P_\mathcal F$ denote the stabilizer in $G$ of the set~$X_{\mathcal F}$.
Then $P_\mathcal F$ is a parabolic subgroup of~$G$ containing~$B$.
In our subsequent considerations, a key role is played by the local structure theorem (see~\cite[Thm.~2.3, Prop.~2.4]{asymp}, \cite[Thm.~1.4]{BLV}), which in our situation may be stated as follows.

\begin{theorem} \label{lst}
Suppose $\mathcal F \subset \mathcal D$ is an arbitrary subset and $P = L \rightthreetimes P_u$ is a Levi decomposition of the group $P = P_\mathcal F$.
Then there exists a closed $L$-stable subvariety $Z \subset X_\mathcal F$ such that the map $P_u \times Z \to X_\mathcal F$ given by the formula $(p,z) \mapsto pz$ is a $P$-equivariant isomorphism, where the action of~$P$ on $P_u \times Z$ is defined by $lu(p,z) = (lupl^{-1}, lz)$ for all $l \in L$, $u,p \in P_u$, $z \in Z$.
Moreover, if $P$ coincides with the stabilizer of the open $B$-orbit in~$X$ then the derived subgroup of~$L$ acts trivially on~$Z$.
\end{theorem}

Below we shall need the following observation.

\begin{proposition} \label{prop_stabilizer}
Suppose $\mathcal F = \mathcal D$ or $\mathcal F = \mathcal D \setminus \lbrace D_0 \rbrace$ where $D_0$ is a color of type~$(T)$.
Then the group $P_\mathcal F$ coincides with the stabilizer of the open $B$-orbit in~$X$.
\end{proposition}

\begin{proof}
If $\mathcal F = \mathcal D$ then the assertion is obvious, therefore in what follows we assume $\mathcal F = \mathcal D \setminus \lbrace D_0 \rbrace$ for a color $D_0$ of type~$(T)$.
Let $Q \supset B$ be an arbitrary minimal parabolic subgroup of~$G$.
Then the condition $QD_0 \ne D_0$ can hold only if $QD' \ne D'$ for some color $D' \in \mathcal D \setminus \lbrace D_0 \rbrace$.
So, if $Q \subset P_\mathcal F$ then $QD_0 = D_0$.
Since $P_\mathcal F$ is generated as a group by all minimal parabolic subgroups contained in it, we obtain $P_\mathcal F D_0 = D_0$, which yields the required result.
\end{proof}

\section{} \label{sect_notation}

Starting from this point and till the end of the paper we assume that $X$ is an affine spherical $G$-variety.
We now introduce some notation.

Fix a maximal torus $T \subset B$ and let $\mathfrak X(T)$ denote its character lattice.
Let $\Delta \subset \mathfrak X(T)$ be the root system of $G$ with respect to~$T$ and let $\Lambda^+ \subset \mathfrak X(T)$ be the monoid of dominant weights with respect to~$B$.

Let $M$ (resp.~$\Gamma$) be the lattice (resp. monoid) of weights of $B$-semiinvariant rational (resp. regular) functions on~$X$.
Since $X$ is affine, we have $M = \ZZ\Gamma$ (see, for instance,~\cite[Prop.~5.14]{Tim}).
Consider the dual lattice $N = \Hom_\ZZ(M,\ZZ)$ and the corresponding rational vector space $N_\QQ = N \otimes_\ZZ \QQ$.
Let $\langle \cdot\,, \cdot \rangle \colon N \times M \to \ZZ$ be the natural pairing.

Since $X$ contains an open $B$-orbit, for every $\lambda \in M$ there exists a unique up to proportionality $B$-semiinvariant rational function $f_\lambda$ on~$X$ of weight~$\lambda$.
Requiring all such functions to take the value~$1$ at a fixed point of the open $B$-orbit, we shall assume that $f_\lambda f_\mu = f_{\lambda + \mu}$ for all $\lambda, \mu \in M$.
Every $D \in \mathcal D^B$ defines an element $\varkappa(D) \in N$ by the formula $\langle \varkappa(D), \lambda \rangle = \ord_D(f_\lambda)$ for all $\lambda \in M$.
Thanks to the normality of~$X$ we have
\begin{equation} \label{eqn_Gamma}
\Gamma = \lbrace \lambda \in M \mid \langle \varkappa(D), \lambda \rangle \ge 0 \ \text{for all} \ D \in \mathcal D^B \rbrace.
\end{equation}
In particular, the set $\lbrace \varkappa(D) \mid D \in \mathcal D^B \rbrace$ generates a strictly convex cone in~$N_\QQ$.

For each strictly convex finitely generated cone $\mathcal C \subset N_\QQ$ let $\mathcal C^1$ denote the set of primitive elements $\rho$ of the lattice~$N$ such that the ray $\QQ_{\ge0}\rho$ is a face of~$\mathcal C$.
For every $\rho \in \mathcal C^1$ define the set
\begin{equation}
\mathfrak R_\rho(\mathcal C) = \lbrace \mu \in M \mid \langle \rho, \mu \rangle = -1; \ \langle \rho',\mu\rangle \ge 0 \ \text{for all} \ \rho' \in \mathcal C^1 \setminus \lbrace \rho \rbrace \rbrace.
\end{equation}
Elements of the set $\mathfrak R(\mathcal C) = \bigsqcup_{\rho \in \mathcal C^1} \mathfrak R_\rho(\mathcal C)$ are called \textit{Demazure roots} of the cone~$\mathcal C$.
Put
\begin{equation}
\Gamma(\mathcal C) = \lbrace \lambda \in M \mid \langle x,\lambda \rangle \ge 0 \ \text{for all} \ x \in \mathcal C \rbrace
\end{equation}
and consider the algebra $A(\mathcal C) = \bigoplus_{\lambda \in \Gamma(\mathcal C)} \KK f_\lambda$.
Below we shall need the following well-known result (see~~\cite[Thm.~2.7]{L1}) providing a description of all $T$-normalized LNDs on~$A(\mathcal C)$.

\begin{theorem} \label{thm_LND_toric}
Let $\mathcal C \subset N_\QQ$ be an arbitrary strictly convex finitely generated cone.
\begin{enumerate}[label=\textup{(\alph*)},ref=\textup{\alph*}]
\item
The set of weights of all $T$-normalized LNDs on $A(\mathcal C)$ equals $\mathfrak R(\mathcal C)$.

\item
For each $\rho \in \mathcal C^1$ and each $\mu \in \mathfrak R_\rho(\mathcal C)$ there exists a unique up to proportionality $T$-normalized LND $\partial_\mu$ on $A(\mathcal C)$ of weight~$\mu$, which is defined by the formula
\begin{equation} \label{eqn_LND_toric}
\partial_\mu(f_\lambda) = \langle \rho, \lambda \rangle f_{\lambda}f_\mu
\end{equation}
for all $\lambda \in \Gamma(\mathcal C)$.
\end{enumerate}
\end{theorem}

\section{} \label{sect_6}

Suppose $H$ is a $B$-root subgroup on~$X$ and $\mathcal F_H = \lbrace D \in \mathcal D \mid HD=D \rbrace$.
Then $H$ preserves the open subset $X_{\mathcal F_H} \subset X$ and thereby defines a $B$-normalized LND on the algebra $\KK[X_{\mathcal F_H}]$.
Note that $\mathcal F_H = \mathcal D$ in the case of vertical or toroidal $H$ and $\mathcal F_H = \mathcal D \setminus \lbrace D_0 \rbrace$ in the case of blurring $H$ moving a color $D_0$ of type~$(T)$.

Now let $\mathcal F = \mathcal D$ or $\mathcal F = \mathcal D \setminus \lbrace D_0 \rbrace$ with $D_0 \in \mathcal D$ being a color of type~$(T)$.
Our goal in this section is to describe all $B$-normalized LNDs on the algebra $\KK[X_\mathcal F]$.

Apply Theorem~\ref{lst} and retain the notation $Z,L,P_u$ used in that theorem.
Then there is a $P$-equivariant isomorphism $X_\mathcal F \simeq P_u \times Z$, via which we shall identify these two varieties below.
Without loss of generality we shall assume $L \supset T$.
In view of Proposition~\ref{prop_stabilizer}, the derived subgroup of~$L$ acts trivially on~$Z$.
Since $X_\mathcal F$ contains an open $B$-orbit, the variety $Z$ contains an open $T$-orbit, which will be denoted by~$Z_0$.
Fix also an arbitrary point $z_0 \in Z_0$.
Let $L_0$ denote the kernel of the action of $L$ on~$Z$ and put $T_0 = T \cap L_0$.
Note that $M$ consists of exactly those characters of~$T$ that restrict trivially to~$T_0$.

For every $\lambda \in M$, the restriction of $f_\lambda$ to the subvariety~$Z$ is a $T$-semiinvariant rational function, which will be still denoted by~$f_\lambda$.
Then $\KK[Z] = \bigoplus_{\lambda \in \Gamma_Z} \KK f_\lambda$ where \begin{equation} \label{eqn_Gamma_Z}
\Gamma_Z = \lbrace \lambda \in M \mid \langle \varkappa(D), \lambda \rangle \ge 0 \ \text{for all} \ D \in \mathcal D^B \setminus \mathcal F \rbrace.
\end{equation}
Without loss of generality we shall assume that $f_\lambda(z_0) = 1$ for all $\lambda \in M$.

Consider the adjoint representation of the group $L$ on the space $\mathfrak p_u = \mathrm{Lie}(P_u)$ and decompose $\mathfrak p_u$ into a direct sum of irreducible $L$-invariant subspaces.
It is well known (see~\cite[Thm.~0.1]{KOS}) that all summands in this decomposition are pairwise non-isomorphic as $L$-modules; let $\Omega \subset \Delta$ be the set of highest weights of these summands with respect to the Borel subgroup $B \cap L \subset L$.
For each $\alpha \in \Omega$ fix a nonzero vector $e_\alpha \in \mathfrak p_u$ of weight~$\alpha$.
The action of $P_u$ on itself by multiplication on the right induces an action of the Lie algebra~$\mathfrak p_u$ on the algebra $\KK[P_u]$; for each $\alpha \in \Omega$ let $\delta_\alpha$ be the derivation of $\KK[P_u]$ determined by the action of~$e_\alpha$.
This derivation is $B$-normalized of weight~$\alpha$, automatically locally nilpotent, and corresponds to the action of the group $\lbrace \exp(te_\alpha) \mid t \in \KK \rbrace$ on $P_u$ by multiplication on the right.
We shall regard $\delta_\alpha$ as an LND on the whole algebra $\KK[P_u \times Z] \simeq \KK[P_u] \otimes_{\KK} \KK[Z]$ by putting $\delta_\alpha(\KK[Z]) = 0$.

For each character $\mu \in \mathfrak X(T)$ put
\begin{equation}
\Omega_\mu = \lbrace \alpha \in \Omega \mid \left.\mu\right|_{T_0} = \left.\alpha\right|_{T_0} \rbrace; \quad \Omega_\mu^0 : = \lbrace \alpha \in \Omega_\mu \mid \mu - \alpha \in \Gamma_Z \rbrace.
\end{equation}
Note that the condition
$\left.\mu\right|_{T_0} = \left.\alpha\right|_{T_0}$ is equivalent to $\mu - \alpha \in M$.

\begin{theorem} \label{thm_LNDs_on_X0}
Every $B$-normalized LND of weight $\mu$ on the algebra $\KK[P_u \times Z]$ has the form
\begin{equation} \label{eqn_LNDs_on_X0}
\sum_{\alpha \in \Omega_\mu^0} c_\alpha f_{\mu - \alpha} \delta_\alpha + \partial_Z
\end{equation}
where $c_\alpha \in \KK$ and $\partial_Z$ is a $T$-normalized LND of weight $\mu$ on~$\KK[Z]$ extended trivially to~$\KK[P_u]$.
Conversely, every derivation on $\KK[P_u \times Z]$ of the above form is $B$-normalized of weight~$\mu$ and locally nilpotent.
\end{theorem}

\begin{proof}
Suppose $\partial$ is a $B$-normalized LND of weight~$\mu$ on $\KK[P_u \times Z]$ and let $\partial_Z$ be the restriction of~$\partial$ to the subalgebra~$\KK[Z]$.
In what follows we regard $\partial_Z$ as a derivation on the whole algebra $\KK[P_u \times Z]$ by putting $\partial_Z(\KK[P_u]) = 0$.
The extension of the derivation $\partial - \partial_Z$ to the algebra $\KK[P_u \times Z_0]$ determines a $B$-semiinvariant vector field $\xi$ of weight~$\mu$ on the smooth variety $P_u \times Z_0$.
Since $B$ acts transitively on $P_u \times Z_0$, $\xi$ is uniquely determined by its value $v$ at the point $(e,z_0)$ where $e \in P_u$ is the identity element.
Since $\partial - \partial_Z$ acts trivially on~$\KK[Z_0]$, it follows that $v$ is a $B \cap L_0$-semiinvariant vector in~$\mathfrak p_u$ of weight $\left.\mu\right|_{T_0}$, therefore $v = \sum_{\alpha \in \Omega_\mu} c_\alpha e_\alpha$ for some $c_\alpha \in \KK$.
On the other hand, observe that the derivation $\sum_{\alpha \in \Omega_\mu} c_\alpha f_{\mu - \alpha}\delta_\alpha$ on $\KK[Z_0]$ is also $B$-semiinvariant of weight~$\mu$ and corresponds to the same tangent vector at $(e,z_0)$, hence it coincides with~$\partial - \partial_Z$.
Since this derivation preserves the algebra $\KK[P_u \times Z]$, the condition $c_\alpha = 0$ should hold for all $\alpha \in \Omega_\mu$ with $\mu - \alpha \notin \Gamma_Z$, which proves the first claim.

Now suppose $\partial$ is a derivation on $\KK[P_u \times Z]$ of the form~(\ref{eqn_LNDs_on_X0}).
Then $\partial$ is automatically $B$-normalized of weight~$\mu$, and it remains to prove that $\partial$ is locally nilpotent.
Since $\KK[P_u]$ is a rational $P_u$-module (with respect to the action on the right), it suffices to check that $\partial$ is locally nilpotent on an arbitrary subspace of the form $V \otimes_\KK \KK[Z]$ where $V \subset \KK[P_u]$ is a finite-dimensional $P_u$-invariant subspace.
As the image of the algebra $\mathfrak p_u$ in $\mathfrak{gl}(V)$ is nilpotent, there exists a flag of subspaces
\begin{equation}
0 = V_0 \subset V_1 \subset V_2 \subset \ldots \subset V_s = V
\end{equation}
with the property $\mathfrak p_u V_i \subset V_{i-1}$ for all $i=1,\ldots, s$.
It follows that for all $i = 1,\ldots, s$, $g \in V_i$, and $f \in \KK[Z]$ we have
\begin{equation}
\partial(gf) = \sum_{\alpha \in \Omega_\mu^0} c_\alpha  \delta_\alpha(g)f_{\mu - \alpha}f + g\partial_Z(f) \in g\partial_Z(f) + V_{i-1} \otimes_\KK \KK[Z]
\end{equation}
because $\delta_\alpha(g) = e_\alpha g \in V_{i-1}$ for all $\alpha \in \Omega$.
Since $\partial_Z$ is an LND on $\KK[Z]$, we obtain $\partial^k(gf) \in V_{i-1} \otimes_\KK \KK[Z]$ for some $k > 0$.
The proof is completed by induction on~$i$.
\end{proof}

\section{} \label{sect_7}

Retain the assumptions and notation of the previous section.
We now study when a $B$-normalized LND on the algebra $\KK[P_u \times Z]$ preserves the subalgebra $\KK[X]$ and thus defines a $B$-root subgroup on the whole variety~$X$.
Let $\lambda \mapsto \overline \lambda$ be an arbitrary projection operator on $\mathfrak X(T) \otimes_\ZZ \QQ$ to the subspace $M \otimes_\ZZ \QQ$.
Let also $\mathcal E_Z \subset N_\QQ$ be the cone generated by the set $\lbrace \varkappa(D) \mid D \in \mathcal D^B \setminus \mathcal F \rbrace$, so that $\Gamma_Z = \Gamma(\mathcal E_Z)$ in view of~(\ref{eqn_Gamma_Z}) and $\KK[Z] = A(\mathcal E_Z)$ (see the notation in \S\,\ref{sect_notation}).

\begin{theorem} \label{thm_constants}
There exists a collection of constants $\lbrace C_D \mid D \in \mathcal F \rbrace$ with the following property: if $\mu \in \mathfrak X(T)$ and $\langle \nu_D, \overline{\mu} \rangle \ge C_D$ for all $D \in \mathcal F$ then every $B$-normalized LND on $\KK[P_u \times Z]$ of weight~$\mu$ preserves~$\KK[X]$.
\end{theorem}

\begin{proof}
Fix a generating system $F_1,\ldots, F_k$ of the algebra $\KK[X]$.
For every $i = 1,\ldots, k$ we have $F_i = \sum_{j=1}^{n_i}g_{ij}f_{\lambda_{ij}}$ for some functions $g_{ij} \in \KK[P_u]$ and weights $\lambda_{ij} \in \Gamma_Z$.
If $\partial$ is an arbitrary derivation of the algebra $\KK[P_u \times Z]$ then $\partial$ preserves $\KK[X]$ if and only if $\ord_D (\partial(F_i)) \ge 0$ for all $D \in \mathcal F$ and $i = 1,\ldots, k$.

In view of Theorem~\ref{thm_LNDs_on_X0}, in order to find the required collection of constants, for each weight $\mu \in \mathfrak X(T)$ it suffices to require that each summand in~(\ref{eqn_LNDs_on_X0}) preserve~$\KK[X]$.

For all $D \in \mathcal F$, $i=1,\ldots,k$, and $\alpha \in \Omega_\mu^0$ with $\delta_\alpha(F_i) \ne 0$ we have
\begin{equation} \label{eqn_ord1}
\ord_D (f_{\mu-\alpha} \delta_\alpha(F_i)) = \ord_D(f_{\mu - \alpha}) + \ord_D (\delta_\alpha(F_i)) = \langle \nu_D, \overline \mu \rangle - \langle \nu_D, \overline \alpha \rangle + \ord_D(\delta_\alpha(F_i)).
\end{equation}

If $\partial_Z \ne 0$ then by Theorem~\ref{thm_LND_toric} we have  $\mu \in \mathfrak R_\rho(\mathcal E_Z)$ for some $\rho \in \mathcal E_Z^1$ (in particular, $\mu \in M$ and $\overline \mu = \mu$) and there exists a nonzero constant $c \in \KK$ such that $\partial_Z(f_\lambda) = c\langle \rho, \lambda \rangle f_{\lambda} f_{\mu}$ for all $\lambda \in \Gamma_Z$.
Then for all $D \in \mathcal F$ and $i = 1,\ldots, k$ with $\partial_Z(F_i) \ne 0$ we have
\begin{multline} \label{eqn_ord2}
\ord_D(\partial_Z(F_i)) = \ord_D(\sum_{j=1}^{n_i}\langle \rho, \lambda_{ij}\rangle g_i f_{\lambda_{ij}}f_\mu) = \langle \nu_D, \overline\mu \rangle + \ord_D(\sum_{j=1}^{n_i}\langle \rho, \lambda_{ij}\rangle g_i f_{\lambda_{ij}}) \ge \\
\langle \nu_D, \overline\mu \rangle + \min\lbrace \ord_D(g_if_{\lambda_{ij}}) \mid j = 1,\ldots, n_i\rbrace.
\end{multline}
It remains to notice that the expressions in~(\ref{eqn_ord1}) and~(\ref{eqn_ord2}) are nonnegative for a suitable choice of the required constants.
\end{proof}

\section{}

Let us deduce several consequences of Theorems~\ref{thm_LNDs_on_X0} and~\ref{thm_constants}.

\begin{corollary}
All $B$-normalized LNDs on $\KK[X]$ of the same weight form a finite-di\-men\-sion\-al vector space over~$\KK$.
\end{corollary}

\begin{proof}
In view of~\cite[Prop.~4.22]{AA}, any two horizontal $B$-root subgroups on~$X$ of the same weight~$\mu$ move the same divisor $D \in \mathcal D^B$.
Consequently, under the conditions of \S\,\ref{sect_6} one can choose a subset $\mathcal F \subset \mathcal D$ such that all $B$-root subgroups on~$X$ of weight~$\mu$ preserve~$X_{\mathcal F}$.
By Theorems~\ref{thm_LNDs_on_X0} and~\ref{thm_LND_toric}, all $B$-normalized LNDs on $\KK[X_{\mathcal F}]$ form a finite-dimensional vector space.
The condition of preserving the subalgebra~$\KK[X]$ determines a subspace in that vector space.
\end{proof}

\begin{corollary}  \label{crl_moves}
Let $D \in \mathcal D^B$ be such that either $D \in \mathcal D^G$ or $D$ is a color of type~$(T)$.
Suppose there exists $\rho \in \mathcal E^1$ such that $\varkappa(D) \in \QQ_{\ge0}\rho$ and $\varkappa(D') \notin \QQ_{\ge0}\rho$ for all $D' \in \mathcal D^B \setminus \lbrace D \rbrace$.
Then there exists a $B$-root subgroup on~$X$ that moves~$D$.
\end{corollary}

\begin{proof}
Put $\mathcal F = \mathcal D$ for $D \in \mathcal D^G$ and $\mathcal F = \mathcal D \setminus \lbrace D \rbrace$ otherwise.
Retain the notation of \S\S\,\ref{sect_6},\ref{sect_7}.
Since $\rho \in \mathcal E_Z$ and $\mathcal E_Z \subset \mathcal E$, we have $\rho \in \mathcal E_Z^1$.
Choose any element $\mu \in \mathfrak R_\rho(\mathcal E_Z)$ and consider the $B$-normalized LND $\partial_\mu$ on $\KK[P_u \times Z]$ of weight~$\mu$ that acts trivially on~$\KK[P_u]$ and by the formula~(\ref{eqn_LND_toric}) on~$\KK[Z]$.
It follows from the hypothesis that there exists $\lambda \in \Gamma$ such that $\langle \rho, \lambda \rangle = 0$ and $\langle \varkappa(D'),\lambda \rangle > 0$ for all $D' \in \mathcal F$.
Then for all integers $N > 0$ we have $N\lambda + \mu \in \mathfrak R_\rho(\mathcal E_Z)$.
By Theorem~\ref{thm_constants} there is a value $N_0$ such that, for all $N \ge N_0$, the LND $\partial_{N\lambda+\mu} = f_{N\lambda}\partial_\mu$ preserves $\KK[X]$ and therefore defines a $B$-root subgroup on~$X$.
This $B$-root subgroup moves~$D$ thanks to~\cite[Prop.~4.22]{AA}.
\end{proof}

In view of~~\cite[Prop.~3.9]{AA}, every $D \in \mathcal D^G$ automatically satisfies the conditions of Corollary~\ref{crl_moves}.
This implies the following result, which was stated as a conjecture in~\cite[Conj.~4.29]{AA}.

\begin{corollary}
For every $D \in \mathcal D^G$ there exists a $B$-root subgroup on~$X$ that moves~$D$.
\end{corollary}

\smallskip

\textbf{Financial support}.
The work of Roman Avdeev was supported by the Russian Science Foundation (grant no.~22-41-02019).
The work of Vladimir Zhgoon was performed within the state assignment for basic scientific research (project no. FNEF-2022-0011) and the HSE University Basic Research Program; it was also partially supported by the Simons Foundation.


\originalsectionstyle

\end{document}